\documentclass[a4paper]{gtart}

\usepackage{psfrag, graphicx, subfigure, epsfig,color}

\usepackage{amsmath, amssymb, latexsym, euscript}

\usepackage{mathptmx, wrapfig}

\def\tri{\mathcal{T}}

\def\bkR{{\rm I\kern-.17em R}}
\def\R{\bkR}
\def\bkZ{{\rm Z\kern-.26em Z}}
\def\Z{\bkZ}

\newcommand{\abs}[1]{\lvert#1\rvert}

\theoremstyle{plain}
\newtheorem{theorem}{Theorem}%[section]
\newtheorem*{theorem*}{Theorem}
\newtheorem{lemma}[theorem]{Lemma}
\newtheorem{proposition}[theorem]{Proposition}
\newtheorem{corollary}[theorem]{Corollary}

\theoremstyle{definition}

\newtheorem*{definition*}{Definition}

\theoremstyle{remark}

\numberwithin{equation}{section}

\begin{document}

\title{Coverings and Minimal Triangulations of 3--Manifolds}
\author{William Jaco, Hyam Rubinstein and Stephan Tillmann}

\begin{abstract}
This paper uses results on the classification of minimal triangulations of 3-manifolds to produce additional results, using covering spaces. Using previous work on minimal triangulations of lens spaces, it is shown that the lens space $L(4k, 2k-1)$ and the generalised quaternionic space $S^3/Q_{4k}$ have complexity $k,$ where $k\ge 2.$ Moreover, it is shown that their minimal triangulations are unique.
\end{abstract}

\makeshorttitle

%%%%%%%%%%%%%%%%%%%%%%%%%%%%%%%

\section{Introduction}

Given a closed, irreducible 3--manifold, its complexity is the minimum number of tetrahedra in a (pseudo--simplicial) triangulation of the manifold. This number agrees with the complexity defined by Matveev~\cite{Mat1990} unless the manifold is $S^3,$ $\R P^3$ or $L(3,1).$ The complexity for an infinite family of closed manifolds has first been given by the authors in \cite{JRT}. The family consisted of lens spaces having a non-trivial $\Z_2$--cohomology class and satisfying an additional, combinatorial constraint.

The main idea in the present paper is the following. Suppose $M$ is a 3--manifold having a connected double cover, $\widetilde{M}.$ A one-vertex triangulation, $\tri,$ of $M$ lifts to a 2--vertex triangulation, $\widetilde{\tri},$ of $\widetilde{M}.$ Because there are two vertices, the lifted triangulation will, in general, not be minimal. One may choose an edge, $\tilde{e},$ joining the two vertices. If certain hypotheses apply, $\tilde{e}$ and the tetrahedra incident with it can be \emph{crushed} to form a new one-vertex triangulation $\widetilde{\tri}^*$ of $\widetilde{M}.$ If $t(\tilde{e})$ denotes the number of tetrahedra incident with $\tilde{e},$ then $c(\widetilde{M}) \le 2 \abs{\tri} -t(\tilde{e}).$ If the complexity of $\widetilde{M}$ is known, this line of argument can be used to show that a given triangulation of $M$ must be minimal. The weakest general bound resulting from this approach is stated below: 

\begin{proposition}\label{pro:crushing-intro}
Let $M$ be a closed, orientable, connected, irreducible 3--manifold, and suppose $\widetilde{M}$ is a connected double cover of $M.$ If $c(M)\ge 2,$ then $c(\widetilde{M}) \le 2\ c(M) -3.$ 
\end{proposition}

This paper determines the minimal triangulations of the manifolds for which equality holds. This is based on previous work \cite{JRT}, where it was shown that $L(2k,1)$ has complexity $2k-3.$ The lens space $L(2k,1)$ double covers both the lens space $L(4k, 2k-1)$ and the generalised quaternionic space $S^3/Q_{4k},$ where $k \ge 2.$

\begin{proposition}\label{pro:crushing-eq}
Let $M$ be a closed, orientable, connected, irreducible 3--manifold, and suppose $\widetilde{M}$ is a connected double cover of $M.$ If $c(\widetilde{M}) = 2\ c(M) -3,$ then either 
\begin{enumerate}
\item $\widetilde{M}= S^3$ and $M= \R P^3,$ or
\item $\widetilde{M}= L(2k,1)$ for some $k\ge 2$ and $M$ has a unique minimal triangulation and is the lens space $L(4k, 2k-1)$ or the generalised quaternionic space $S^3/Q_{4k}.$ 
\end{enumerate}
\end{proposition}

It should be noted that the proof does not use the fact that the minimal triangulation of $L(2k,1)$ is unique; the uniqueness part follows from the fact that these triangulations are shown to be dual to one-sided Heegaard splittings. We now describe the unique minimal triangulations in an alternative way.

Recall from \cite{JR:LT} that each lens space has a unique \emph{minimal layered triangulation} and that this is conjectured to be its unique minimal triangulation. The minimal layered triangulation of the lens space $L(4k, 2k-1)$ has $k$ tetrahedra. (The main result in \cite{JRT} does not include these lens spaces.) 

Following Burton \cite{bab}, a \emph{layered chain} of length $k,$ denoted $C_k,$ is defined to be a certain triangulation of the solid torus with four boundary faces and $k$ tetrahedra. A suitable identification of the boundary faces of $C_k$ results in the \emph{twisted layered loop triangulation} $\widehat{C}_k$ of $S^3/Q_{4k}.$ 

\begin{corollary}\label{cor:minimal triangulations}
For every $k \ge 2,$ $L(4k, 2k-1)$ and $S^3/Q_{4k}$ have complexity $k.$ The unique minimal triangulation of $L(4k, 2k-1)$ is its minimal layered triangulation and the unique minimal triangulation of $S^3/Q_{4k}$ is its twisted layered loop triangulation.
\end{corollary}

This implies that for every positive integer $k,$ there is a closed, orientable, connected, irreducible 3--manifold of complexity $k.$ Since $S^3,$ $L(4,1)$ and $L(5,2)$ have complexity one, we in fact have:

\begin{corollary}\label{cor:amusing}
For every positive integer $k,$ there are at least two spherical 3--manifolds of complexity $k.$
\end{corollary}

The first author is partially supported by NSF Grant DMS-0505609 and the Grayce B. Kerr Foundation. The second and third authors are partially supported under the Australian Research Council's Discovery funding scheme (project number DP0664276).

%%%%%%%%%%%%%%%%%%%%%%%%%%%%%%%
%\newpage
%%%%%%%%%%%%%%%%%%%%%%%%%%%%%%%

%%%%%%%%%%%%%%%%%%%%%%%%%%%%%%%
%\newpage
%%%%%%%%%%%%%%%%%%%%%%%%%%%%%%%

\section{Lifting and crushing}
\label{sec:layered triangulations}

We use the same notation as in \cite{JRT} for triangulations and for the standard models of low degree edges in minimal triangulations.

\begin{lemma}\label{lem:edges}
Suppose that the minimal triangulation $\tri$ of the closed, orientable, connected and irreducible 3--manifold $M$ is lifted to a triangulation $\widetilde{\tri}$ of a connected double cover. Assume $c(M)\ge 4.$ Then every edge which connects the two distinct vertices in $\widetilde{\tri}$ and which is contained in at most three distinct tetrahedra is contained in precisely three tetrahedra. Moreover, its image in $\tri$ is modelled on the edge of degree four in the complex $X^1_{4;3}.$
\end{lemma}

\begin{figure}[t]
\psfrag{e}{{\small $e$}}
\begin{center}
      \includegraphics[height=2.7cm]{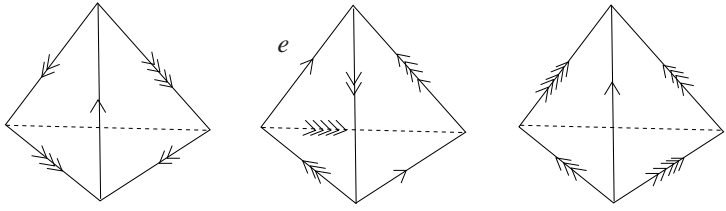}
\end{center}
    \caption{The complex $X^1_{4;3}\cong$ solid torus}
         \label{fig:X31}
\end{figure}

\begin{proof}
Since $c(M) \ge 4,$ we have $M\neq S^3, \R P^3, L(3,1)$ and $\tri$ is also $0$--efficient and has a single vertex. Hence $\widetilde{\tri}$ has precisely two vertices. Denote $\tilde{e}$ an edge in $\widetilde{\tri}$ with distinct endpoints. Suppose $\tilde{e}$ is contained in at most three tetrahedra. Then the same is true for its image $e$ in $\tri.$ Moreover, there is a non-trivial homomorphism $\varphi\co \pi_1(M)\to \Z_2$ associated with the covering, and $\varphi[e]=1.$

First note that if the degree of $e$ is at most five, then inspection of the possibilities stated in \cite{JRT} ---keeping in mind that $\varphi[e]=1,$ $c(M)\ge 4$ and $e$ is incident with at most three tetrahedra--- yields the possibilities $X^2_{4;2},$ $X^1_{4;3},$ $X^1_{5;3}$ and $X^2_{5;3}.$ Recall that $t(\tilde{e})$ denotes the number of tetrahedra incident with $\tilde{e}.$ The last two possibilities force $t(\tilde{e})>3,$ a contradiction. In case it is modelled on $X^2_{4;2},$ one observes that $\tilde{e}$ is of degree five and contained in precisely four tetrahedra in $\widetilde{\tri};$ a contradiction. This leaves the complex $X^1_{4;3}$ shown in Figure \ref{fig:X31} in this case. Either $\varphi[e] = \varphi[e_2] = \varphi[e_5]=1$ and $\varphi[e_3]=\varphi[e_4]=0$ or $\varphi[e] = \varphi[e_3] = \varphi[e_4]=1$ and $\varphi[e_2]=\varphi[e_5]=0,$ where the subscript corresponds to the number of arrows.

\begin{figure}[t]
\begin{center}
    \subfigure[Type 1]{
      \includegraphics[height=2.3cm]{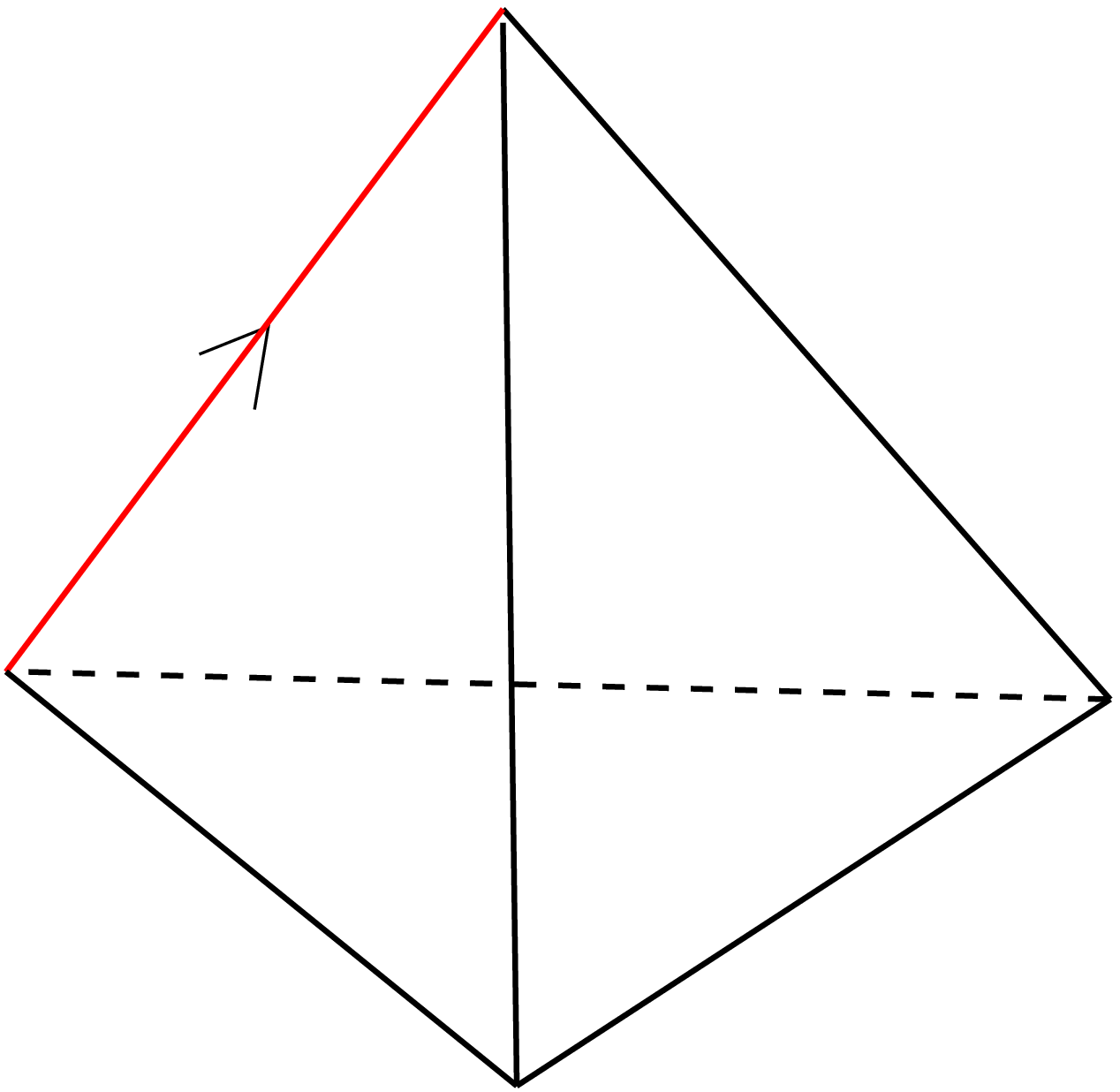}
    } 
    \qquad\qquad
    \subfigure[Type 2a]{
      \includegraphics[height=2.3cm]{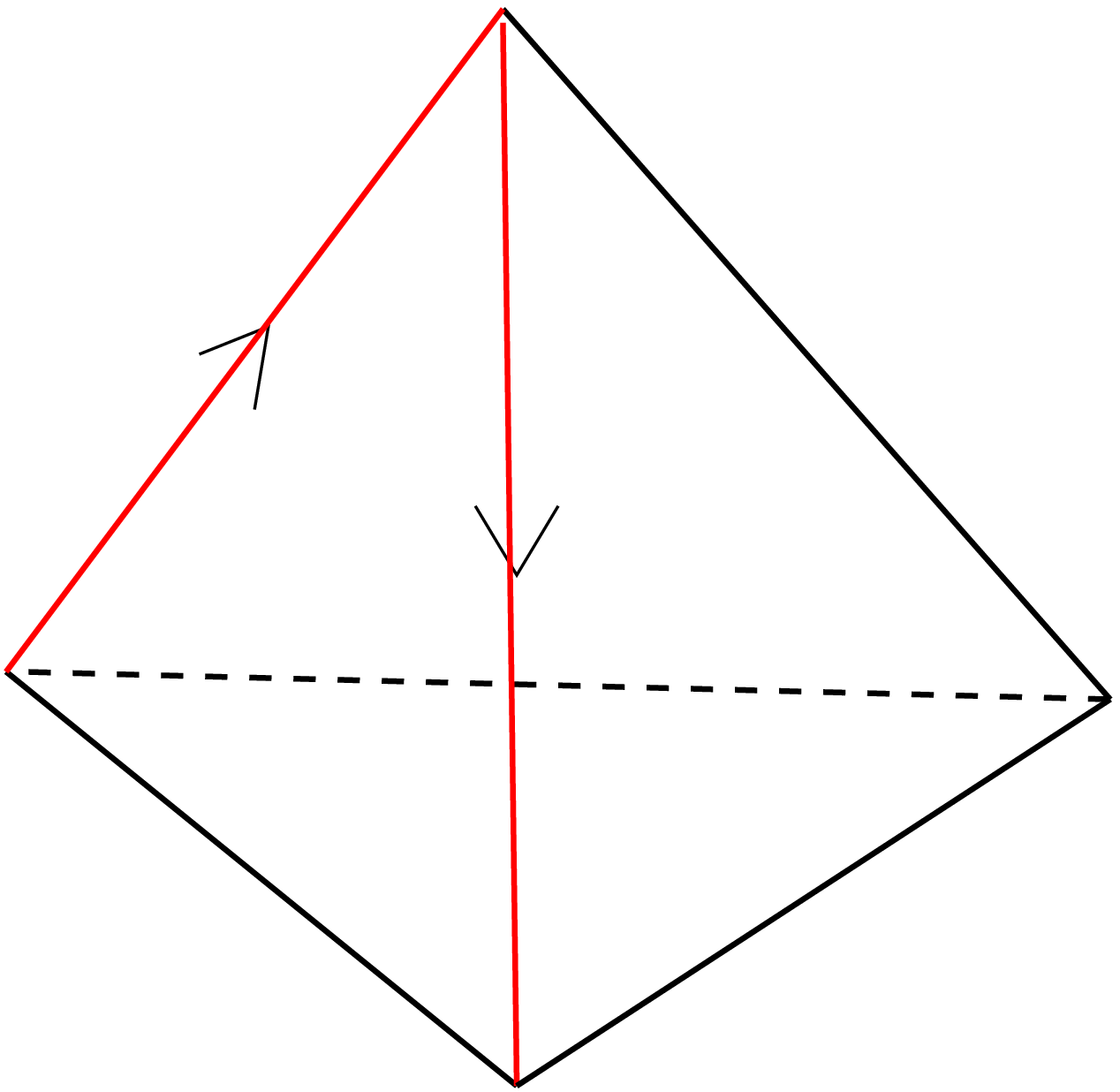}
    } 
    \qquad\qquad
    \subfigure[Type 2b]{
      \includegraphics[height=2.3cm]{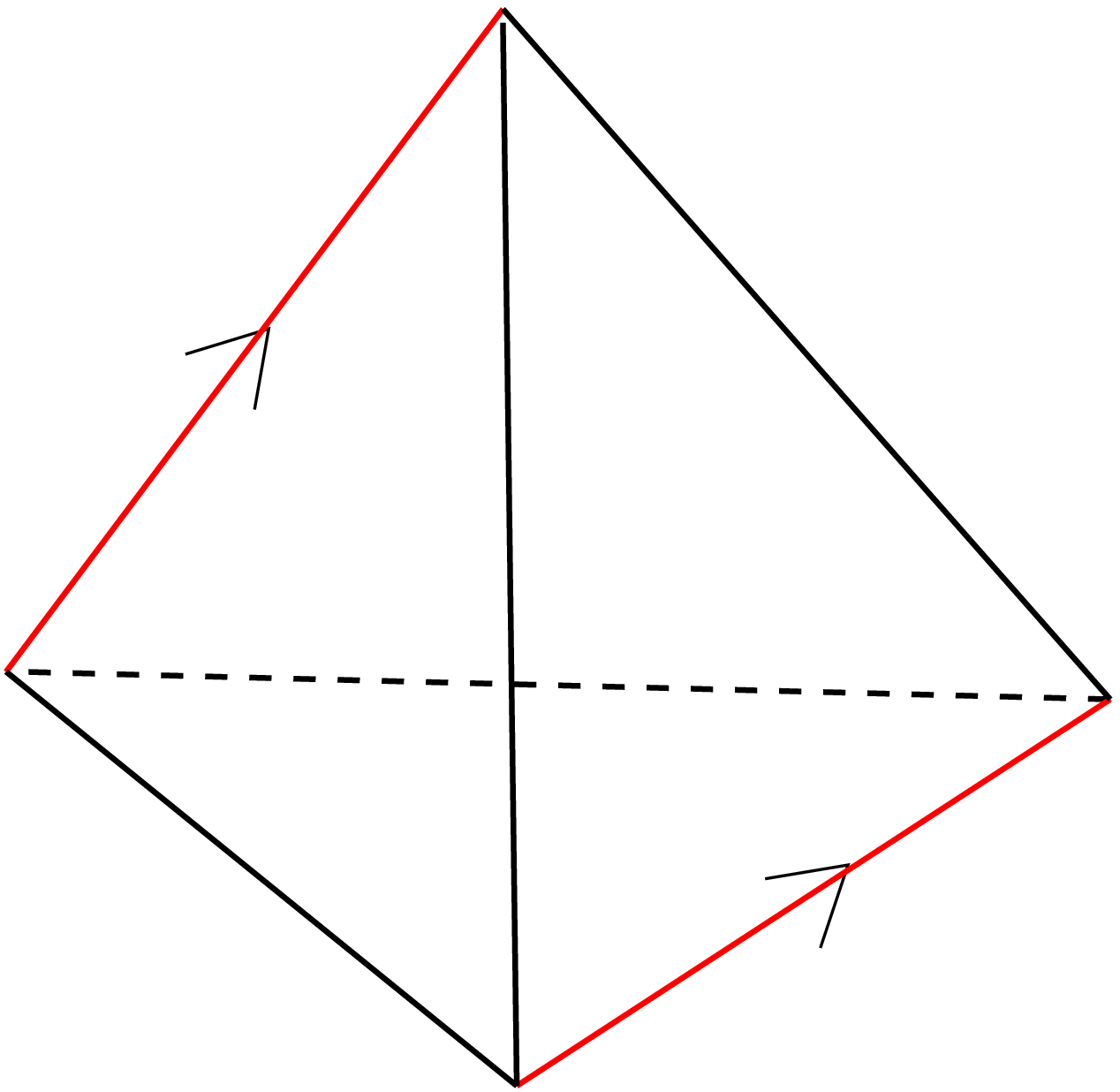}
    } 
    \qquad\qquad
    \subfigure[Type 3a]{
      \includegraphics[height=2.3cm]{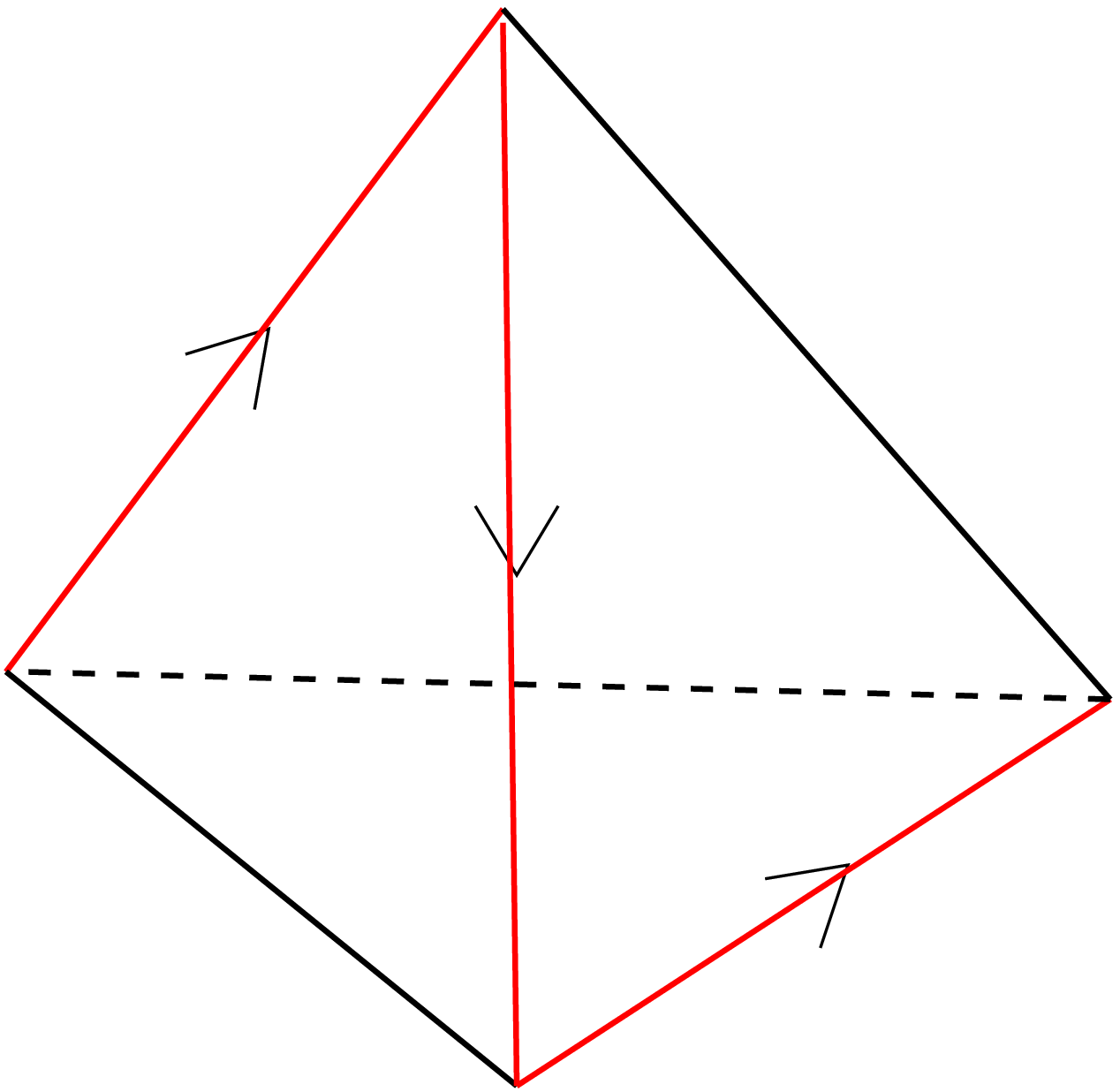}
    } 
    \qquad\qquad
    \subfigure[Type 3b]{
      \includegraphics[height=2.3cm]{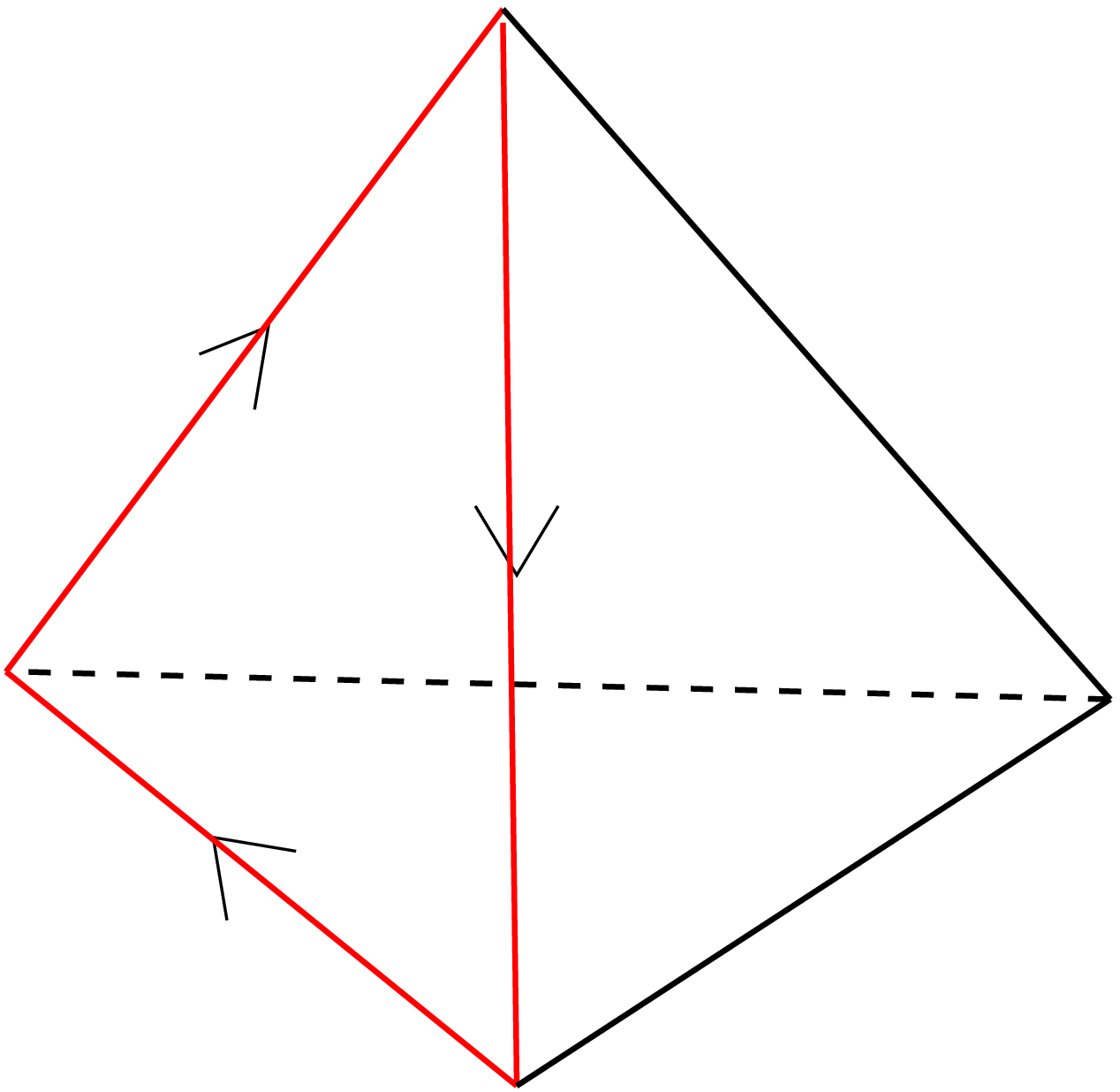}
    } 
    \qquad\qquad
    \subfigure[Type 4]{
      \includegraphics[height=2.3cm]{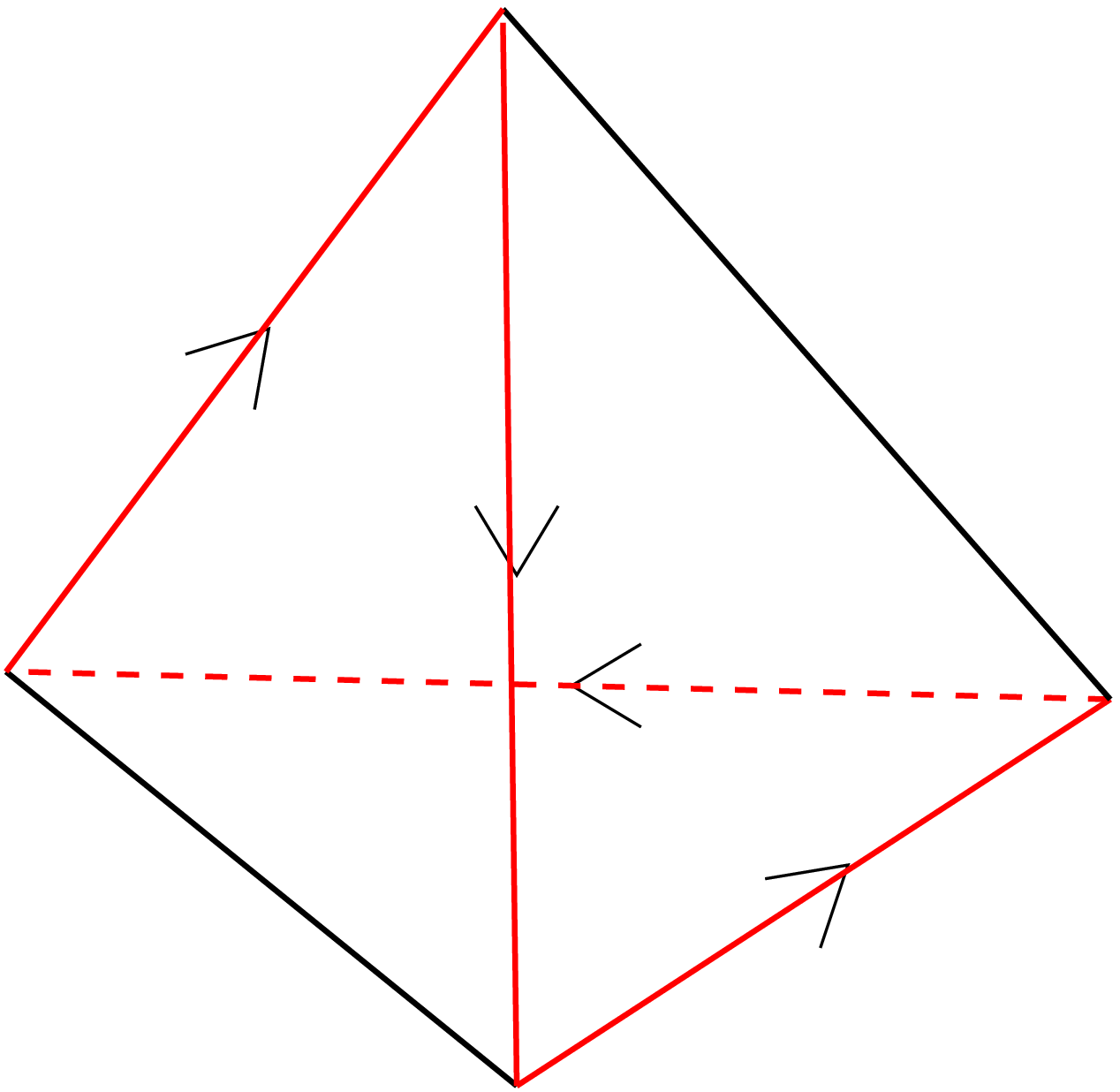}
    } 
\end{center}
    \caption{The possibilities for $p^{-1}(e) \cap \sigma$}
     \label{fig:edges}
\end{figure}

It remains to analyse the possibilities when $d(e)\ge 6.$ We make some preliminary observations that limit the number of cases to consider. Since $M\neq S^3$ and $\tri$ is 0--efficient, it follows that no face in $\tri$ is a cone \cite{JR} or a dunce hat \cite{JRT}. In particular, if $\sigma$ is a 3--simplex in $\widetilde{\Delta},$ then $p^{-1}(e) \cap \sigma$ consists of at most four edges, and the possibilities (up to combinatorial equivalence) are shown in Figure \ref{fig:edges}. Note that type 3b is not possible, since $1+1+1 \neq 0$ in $\Z_2.$ For each tetrahedron of type 2a, 3a or 4, we have that each of its two lifts to $\widetilde{\tri}$ is incident with $\tilde{e}.$ For each of the others, at least one of the lifts is incident with $\tilde{e}.$

Since $c(M)\ge 4,$ at least one face of the tetrahedra incident with $e$ does not have $e$ as an edge. So at least one of the tetrahedra is of type 1 or 2a. Since $d(e)\ge 6,$ the only case with two tetrahedra is (2a,4); this is not possible as the faces incident with $e$ cannot be matched in pairs. In case there are three tetrahedra, no tetrahedron can be of type 2a, 3a or 4 as otherwise $t(\tilde{e})>3.$ This leaves no possibility when $d(e)\ge 6.$
\end{proof}

If $M$ has complexity $2$ or $3$ and $\widetilde{M}$ is a connected double cover of $M,$ then the inequality $c(\widetilde{M}) \le 2 c(M)-3$ holds by inspection of the census of \cite{bab}. To streamline notation, we will list the possibilities as $\{\widetilde{M}, c(\widetilde{M}); M, c(M)\}.$ They are: $\{S^3, 1; \R P^3, 2\},$ $\{L(3,1), 2; L(6,1), 3\},$ $\{L(5,1), 2; L(10,3), 3\}$ and the cases $k=2,3$ in the families $\{L(2k,1), 2k-3; L(4k,2k-1), k\}$ and $\{L(2k,1), 2k-3; S^3/Q_{4k}, k\}.$

Proposition \ref{pro:crushing-intro} is implied by this discussion and the following.

\begin{proposition}\label{pro:crushing}
Suppose that the minimal triangulation $\tri$ of the closed, orientable, connected and irreducible 3--manifold $M$ is lifted to a triangulation $\widetilde{\tri}$ of a connected double cover, and that $c(M) \ge 4.$

Then every edge connecting the two distinct vertices in $\widetilde{\tri}$ is contained in at least three distinct tetrahedra and can be crushed. In particular, if $\tilde{e}$ is such an edge, then
$$c(\widetilde{M}) \le 2 c(M) -t(\tilde{e}) \le 2 c(M)-3,$$
where $t(\tilde{e})$ is the number of tetrahedra incident with $\tilde{e}.$
\end{proposition}

\begin{proof}
Since $c(M) \ge 4,$ $\tri$ is $0$--efficient with a single vertex, and hence $\widetilde{\tri}$ has precisely two vertices. Let $\tilde{e}$ be an edge in $\widetilde{\tri}$ with distinct endpoints. It follows from Lemma \ref{lem:edges} that $t(\tilde{e})\ge 3.$ As in \cite{JR}, we can crush $\tilde{e}$ and the surrounding tetrahedra to form a one vertex triangulation triangulation of $\widetilde{M}$, so long as there are no inadmissible gluings on the boundary of this set of tetrahedra. That this is always the case follows from the fact that the two ends of $\tilde{e}$ are the two vertices, $v$ and  $v^\prime,$ of the triangulation. So any edge $\tilde{e}^\prime$ in a face $f$ containing $\tilde{e}$ cannot be glued to the third edge $\tilde{e}^*$ of $f$, since if $\tilde{e}^\prime$ has ends at $v,v$ then $\tilde{e}^*$  has ends at $v, v^\prime$ and vice versa. Moreover, neither $\tilde{e}^\prime$ nor $\tilde{e}^*$ can be glued to $\tilde{e}$ since otherwise the image of $f$ in $\tri$ is a cone or dunce hat which implies $M = S^3$ since $\tri$ is 0--efficient (see \cite{JR} Corollary~5.4 and \cite{JRT} Lemma~7). Hence $\tilde{e}$ can be crushed and we have $c(\widetilde{M}) \le 2 c(M) -t(\tilde{e}) \le 2 c(M)-3.$
\end{proof}

%%%%%%%%%%%%%%%%%%%%%%%%%%%%%%%

\begin{proof}[Proof of Proposition \ref{pro:crushing-eq}]

Suppose the hypothesis of the proposition is satisfied. If $c(M)\le 3,$ the statement follows from the discussion preceding Proposition \ref{pro:crushing}. Hence assume $c(M)\ge 4$ and choose a minimal triangulation, $\tri,$ of $M.$ Then every edge connecting the two vertices, $v$ and $v',$ of $\widetilde{\tri}$ can be crushed. Since we stipulate equality, it follows that every such edge is contained in precisely three distinct tetrahedra. Hence its image under the covering map is contained in at most three distinct tetrahedra in $\tri.$ Denote $e$ the image in $\tri$ of an edge in $\widetilde{\tri}$ connecting the two vertices.
It follows from Lemma \ref{lem:edges} that $e$ is of degree four and its neighbourhood is modelled on $X^1_{4;3}.$ Note that each of its lifts, $\tilde{e}_i,$ to $\widetilde{\tri}$ also has its neighbourhood modelled on $X^1_{4;3}.$ Moreover, each tetrahedron incident with $\tilde{e}_i$ has precisely one edge with both ends at $v$ and one with both ends at $v'.$ Since each tetrahedron also contains other edges than $\tilde{e}_i$ connecting the two vertices, this propagates and one observes that in $\widetilde{\tri}$ there is precisely one edge with both ends at $v$ and precisely one edge with both ends at $v'.$ Moreover, there is a normal surface, $S,$ made up entirely of quadrilateral discs which is the boundary of a neighbourhood of each of these edges. Since $\widetilde{M}$ is orientable, this implies that $S$ is a torus. (Alternatively, observe that $S$ is separating and has vanishing Euler characteristic since all vertices in the cell decomposition of $S$ by quadrilateral discs have degree four.) Whence $\widetilde{M}$ is a lens space. Moreover, $\tri$ contains a quadrilateral surface which is double covered by the torus and dual to the $Z_2$--cohomology class. It hence is a Klein bottle and incompressible. In particular, the triangulation $\tri$ is dual to a 1--sided Heegaard diagram.

The regular neighbourhood of the Klein bottle is homeomorphic to the twisted $I$--bundle over the Klein bottle, and its boundary is hence a torus. As in \cite{Rubin1979}, choose generators $a,b$ for the Klein bottle such that $a, b^2$ correspond to standard generators for the boundary torus. Then $M$ is obtained by attaching a solid torus with meridian disc corresponding to the curve $b^{2m}a^n$ for some $m$ and $n.$ The dual triangulation has precisely one tetrahedron for each intersection point of the boundary of the meridian disc. The minimal number of such points is $mn$ and there is a unique curve up to isotopy which realises this. Hence $\tri$ is the unique minimal triangulation of $M.$

Moreover, the cover $\widetilde{M}$ can now be identified as $L(2mn,x),$ where $x=1-2np=-1-2mq,$ where $(p,q)$ are chosen such that $pn-qm=1.$ Note that $\widetilde{M}$ is not uniquely determined by this. However, we know by assumption that $c(L(2mn,x))=2mn-3.$ This forces $x=1,$ since otherwise $c(L(2mn,x))<2mn-3$ as can be seen from the number of tetrahedra in the minimal layered triangulation of $L(2mn,x),$ see \cite{JR:LT}. But then $n=1$ or $m=1,$ which gives the conclusion of the proposition.
\end{proof}

\begin{proof}[Proof of Corollary \ref{cor:minimal triangulations}]

It is shown in the above proof that for every $k \ge 2,$ $S^3/Q_{4k}$ and $L(4k, 2k-1)$ have a 1--sided Heegaard diagram with precisely $k$ intersection points, and hence a triangulation having precisely $k$ tetrahedra. The statement of the corollary holds for $k=2, 3, 4$ by inspection of the census in \cite{bab}. Hence assume $k \ge 5.$ Then, by inspection of the census, we have $c(S^3/Q_{4k})\ge 4$ and $c(L(4k, 2k-1))\ge 4.$

Suppose a minimal triangulation of $L(4k, 2k-1)$ or $S^3/Q_{4k}$ has at most $k-1$ tetrahedra. Lifting to the double covering $L(2k,1)$, we get the triangulation $\widetilde{\tri}$ with two vertices and at most $2k-2$ tetrahedra. Proposition \ref{pro:crushing} implies that any edge connecting the two vertices can be crushed and must belong to at least three tetrahedra. So crushing such an edge gives a one-vertex triangulation of $L(2k,1)$ with at most $2k-5$ tetrahedra, giving a contradiction to the fact that $L(2k,1)$ has complexity $2k-3.$ Hence both $L(4k, 2k-1)$ and $S^3/Q_{4k}$ have complexity $k$ and the manifolds satisfy the hypothesis of Proposition \ref{pro:crushing-eq}.

The unique minimal triangulations have been described via the dual 1--sided Heegaard diagram; the alternative descriptions stated in Corollary \ref{cor:minimal triangulations} are given in the next subsections.
\end{proof}

%%%%%%%%%%%%%%%%%%%%%%%%%%%%%%%

\subsection{The minimal layered triangulation of $L(4k,2k-1)$}

For the lens space $M=L(4k,2k-1),$ the minimal layered triangulation, $\tri,$ is obtained from the minimal layered extension of $\{ 2, 2k-1, 2k+1\}$ by folding along $2,$ see~\cite{JR:LT} for details. The sequence of labelings of the layered triangulation is:
$$(2,1,1), (3,2,1), (5,3,2), (7,5,2),\break  (9,7,2), \dots, (2k+1,2k-1,2).$$
The minimal layered triangulation of $L(4k,2k-1)$ has therefore $k$ tetrahedra, and hence is the unique minimal triangulation.

%%%%%%%%%%%%%%%%%%%%%%%%%%%%%%%
%\newpage
%%%%%%%%%%%%%%%%%%%%%%%%%%%%%%%

\subsection{The twisted layered loop triangulation}

Note that 
$$M_k = S^3/Q_{4k} = S^2(\ (2,1), (2,1), (k,1-k) \ ) = S^2(\ (1,-1), (2,1),(2,1),(k,1)\ ),$$
the latter being the unique normal form. Moreover, 
$$ \pi_1(M_k) = Q_{4k} \cong \langle x, y \ |\ xyx^{-1} = y^{-1}, x^2 = y^k \rangle.$$
Element $x$ has order $4,$ $y$ has order $2k.$ The subgroup $\langle y \rangle$ has index two, hence it is normal, and $Q_{4k}$ has order $4k.$ It follows that  $H_1(M_k) \cong \Z_4$ if $k$ is odd, and $H_1(M_k) \cong \Z_2 \oplus \Z_2$ if $k$ is even. The double cover of $M_k$ associated to the action of $\langle y \rangle \cong \Z_{2k}$ on $S^3$ is a lens space; in fact $S^3/\langle y \rangle=L(2k,1).$ 

\begin{figure}[t]
\psfrag{i}{{\small $e_2$}}
\psfrag{h}{{\small $e_1$}}
\psfrag{b}{{\small $b$}}
\psfrag{t}{{\small $t$}}
\psfrag{e}{{\small $e_h$}}
\psfrag{f}{{\small $e_{h+1}$}}
\psfrag{g}{{\small $e_{h+2}$}}
\begin{center}
      \includegraphics[height=3.5cm]{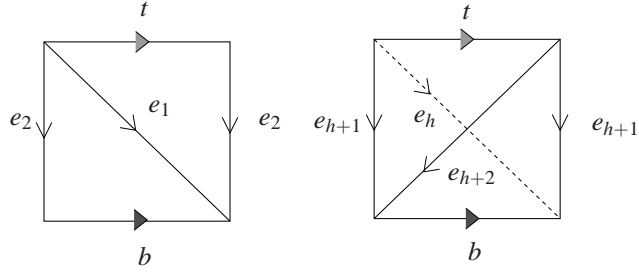}
\end{center}
    \caption{The twisted layered loop triangulation}
     \label{fig:quaternionic}
\end{figure}

The starting point for the twisted layered loop triangulation is the triangulation with two faces of the annulus shown with labelling in Figure \ref{fig:quaternionic}. The edges corresponding to the two boundary components are denoted $t$ for \emph{top} and $b$ for \emph{bottom}, and oriented so that they correspond to the same element in fundamental group. The remaining two edges are $e_1$ and $e_2,$ oriented from $t$ to $b.$ Tetrahedron $\sigma_1$ is layered along $e_1,$ and the new edge denoted $e_3$ and oriented from $t$ to $b.$ The annulus is thus identified with two faces of $\sigma_1.$ Inductively, tetrahedron $\sigma_h$ is layered along edge $e_h,$ and the new edge $e_{h+2}$ is oriented from $t$ to $b.$ Assume $k$ tetrahedra have thus been attached; if $k=0$ we have an annulus, if $k=1$ a creased solid torus and if $k \ge 2$ a solid torus. Denote the resulting triangulation $C_k.$

Then the two free faces of tetrahedron $\sigma_k$ in $C_k$ are identified with the two free faces of tetrahedron $\sigma_1$ such that $\sigma_k$ is layered along $e_1$ with $e_1 \leftrightarrow -e_{k+1},$ $e_2 \leftrightarrow -e_{k+2}$ and $t\leftrightarrow-b.$ The result is a closed 3--manifold, denoted $M_k,$ and the triangulation, denoted $\widehat{C}_k$ is termed its \emph{twisted layered loop triangulation}.

The following result can be found in Burton's thesis (\cite{bab}, Theorem 3.3.11).

\begin{proposition}[Burton]
For each $k\ge 1,$ 
$$M_k = S^3/Q_{4k} = S^2(\ (2,1), (2,1), (k,1-k) \ ).$$
\end{proposition}
\begin{proof} 
Place a quadrilateral in each tetrahedron separating edges $t$ and $b.$ This gives a one-sided Klein bottle, $S_1,$ in $M_k,$ and $\overline{M_k \setminus S_1}$ is a solid torus with core $t=-b.$ We thus have a one-sided Heegaard splitting of non-orientable genus two. Work in \cite{Rubin1979} by the second author identifies such manifolds using the meridian of $\overline{M_k \setminus S_1},$ giving $M_k = S^3/Q_{4k}.$ It is shown by Orlik \cite{Or1972} that $S^3/Q_{4k} = S^2(\ (2,1), (2,1), (k,1-k) \ ).$ 
\end{proof}

%%%%%%%%%%%%%%%%%%%%%%%%%%%
%\newpage
%%%%%%%%%%%%%%%%%%%%%%%%%%%

%%%%%%%%%%%%%%%%%%%%%%%%%%%

\address{Department of Mathematics, Oklahoma State University, Stillwater, OK 74078-1058, USA}
\email{jaco@math.okstate.edu}

\address{Department of Mathematics and Statistics, The University of Melbourne, VIC 3010, Australia} 
\email{rubin@ms.unimelb.edu.au} 

\address{Department of Mathematics and Statistics, The University of Melbourne, VIC 3010, Australia} 
\email{tillmann@ms.unimelb.edu.au} 
\Addresses
                                                      
\end{document}